\documentclass[preprint,1p]{elsarticle}

\makeatletter
\def\ps@pprintTitle{%
	\let\@oddhead\@empty
	\let\@evenhead\@empty
	\def\@oddfoot{\centerline{\thepage}}%
	\let\@evenfoot\@oddfoot}
\makeatother
\usepackage[unicode]{hyperref}

\usepackage{latexsym}
\usepackage{indentfirst}
\usepackage{amsxtra}
\usepackage{amssymb}
\usepackage{amsthm}
\usepackage{amsmath}
\usepackage{mathrsfs} 

\usepackage[capitalise]{cleveref}

\newtheorem{theor}{Theorem}
\newtheorem{prop}[theor]{Proposition}
\newtheorem{lemma}[theor]{Lemma}
\newtheorem{cor}[theor]{Corollary}
\theoremstyle{definition}               
\newtheorem{defin}[theor]{Definition}
\newtheorem{ex}[theor]{Example}
\newtheorem{exs}[theor]{Examples}
\newtheorem{remk}[theor]{Remark}

\DeclareMathOperator{\id}{id}

\DeclareMathOperator{\sgn}{sgn}

\begin{document}
\begin{frontmatter}
	\title{Set-theoretical solutions of the pentagon equation on groups}
	\tnotetext[mytitlenote]{{This work was partially supported by the Dipartimento di Matematica e Fisica ``Ennio De Giorgi" - Universit\`{a} del
			Salento. The authors are members of GNSAGA (INdAM).}}
	\author [unile] {Francesco~CATINO\corref{cor1}}
	\ead{francesco.catino@unisalento.it}
	\cortext[cor1]{Corresponding author}
	\author [unile] {Marzia~MAZZOTTA}
	\ead{marzia.mazzotta@unisalento.it}
	\author [unile] {Maria Maddalena~MICCOLI}
	\ead{maddalena.miccoli@unisalento.it}
	\address[unile]{Dipartimento di Matematica e Fisica ``Ennio De Giorgi"\\
		Universit\`{a} del Salento\\
		Via Provinciale Lecce-Arnesano \\
		73100 Lecce (Italy)\\}
	\begin{abstract}
		Let $M$ be a set. A \emph{set-theoretical solution of the pentagon equation} on $M$ is a map $s:M\times M\longrightarrow M\times M$ such that
		\begin{equation*}
		s_{23}\, s_{13}\, s_{12}=s_{12}\, s_{23},
		\end{equation*}
		where $s_{12}=s\times \id_M$,  $s_{23}=\id_M\times s$ and  $s_{13}=(\id_M\times \tau)s_{12}(\id_M\times \tau)$, and $\tau$ is the flip map, i.e., the permutation on $M\times M$ given by $\tau(x,y)=(y,x)$, for all $x,y\in M$.\\
		In this paper we give a complete description of the set-theoretical solutions of the form $s(x,y)=(x\cdot y , x\ast y)$ when either $(M,\cdot)$ or $(M,\ast)$ is a group; moreover, we raise some questions.
	\end{abstract} 
	
\end{frontmatter}

\section{Introduction}

\noindent Let $V$ be a vector space over a field $F$. A linear map $S:V\otimes V\rightarrow V\otimes V$ is said to be a \emph{solution of the pentagon equation} if 
\begin{equation*}\label{P1}
S_{12}S_{13}S_{23}=S_{23}S_{12},
\end{equation*}
where $S_{12}=S\otimes \id_V$, $S_{23}=\id_V\otimes S$ and $S_{13}=(\id_V\otimes T)\; S_{12}\;(\id_V\otimes T)$, where $T$ is the \emph{flip map} on $V\otimes V$, i.e., $T(v\otimes w)=w\otimes v$, for all $v,w\in V$.\\
Solutions of the pentagon equation appear in various contexts and with different terminologies. For instance, the canonical element in the Heisenberg double of an arbitrary Hopf algebra satisfies the pentagon equation \cite{Ka96}. If $\mathcal{H}$ is a Hilbert space, a unitary operator from $\mathcal{H}\otimes\mathcal{H}$ into itself is said to be multiplicative if it is a solution of the pentagon equation \cite{BaSk93}. According to Street \cite{Str98}, fixed a braided monoidal category $\mathcal{V}$ in the sense of Joyal-Street, an arrow $V:A \otimes A \rightarrow A \otimes A$ in $\mathcal{V}$ is said to be a fusion operator if it satisfies the pentagon relation. 
For more information about relationships between of the pentagon equation and other topics, see the recent paper of Dimakis and M\"uller-Hoissen \cite{DiMu15} along with the references therein. \\ 
A \emph{set-theoretical solution of the pentagon equation} on an arbitrary set $M$ is a map $s:M\times M\rightarrow M\times M$ which satisfies the ``reversed'' pentagon equation
\begin{equation*}
s_{23}\, s_{13}\, s_{12}=s_{12}\, s_{23},
\end{equation*}
where $s_{12}=s\times \id_M$,  $s_{23}=\id_M\times s$ and  $s_{13}=(\id_M\times \tau)\,s_{12}\,(\id_M\times \tau)$, and $\tau$ is the flip map, i.e., the permutation on $M\times M$ given by $\tau(x,y)=(y,x)$, for all $x,y\in M$.\\ 
A link between solutions and set-theoretical solutions of the pentagon equation is highlighted in the paper of Kashaev and Reshetikhin \cite{KaRe07}. If $M$ is a finite set,  $F$ a field and $V:=F^M$ the vector space of the functions from $M$ to $F$, then it is well-known that $V\otimes V$ is isomorphic to $F^{M\times M}$. Moreover, for any map $s:M\times M\rightarrow M\times M$ one can associate its pull-back $S$, i.e., the linear map $S:V\otimes V\rightarrow V\otimes V$ given by
\begin{equation*}
(S\varphi)(x,y)= \varphi(s(x,y)),
\end{equation*}
for every $\varphi\in F^{M\times M}$. Then, $S$ is a solution of the pentagon equation if and only if $s$ is a set-theoretical solution of the pentagon equation.\\
Recently, there have appeared several papers on set-theoretical solutions. See, for example, Zakrzewski \cite{Za92}, Baaj and Skandalis \cite{BaSk93,BaSk98}, Kashaev and Sergeev \cite{KaSe98}, Jiang and Liu \cite{JiLi05}, Kashaev and Reshetikhin \cite{KaRe07}, and Kashaev \cite{Ka11}.

For a map $s:M\times M\rightarrow M\times M$ define binary operations $\cdot$ and $\ast$ via
\begin{equation*}
s(x,y)=(x\cdot y\; , \; x\ast y), 
\end{equation*}
for all $x,y \in M$. One can easily see that the map $s$ is a set-theoretical solution of the pentagon equation if and only if the following conditions hold
	\begin{enumerate}
		\item[(1)] $(x \cdot y) \cdot z=x \cdot (y \cdot z)$,
		\item[(2)] $(x \ast y) \cdot ((x \cdot y) \ast z)=x \ast (y \cdot z)$,
		\item[(3)] $(x \ast y) \ast ((x \cdot y) \ast z)=y \ast z$,
	\end{enumerate}
	for all $x,y,z \in M$.\\
As noted by Kashaev and Sergeev \cite{KaSe98}, under the assumption that  $(M,\cdot)$ is a group, one obtains that the only \emph{invertible} set-theoretical solution $s$ of the pentagon equation on $M$ is given by $s(x,y)=(x\cdot y,\; y)$. \\
In this paper we give a complete description of all set-theoretical solutions of the pentagon equation of the form $s(x,y)=(x\cdot y\; , \; x\ast y)$ when either $(M,\cdot)$ or $(M,\ast)$ is a group. Finally, in the last section some questions are raised and partial answers are given to these.

\section{Basic results}
In this section, for the ease of the reader, we fix some notations and terminologies. Furthermore, we put together some available results and examples in the literature of set-theoretical solutions of the pentagon equation, without adding any new results.

\begin{defin}
Let $M$ be a set. A \emph{set-theoretical solution of the pentagon equation} on $M$ is a map $s:M\times M\longrightarrow M\times M$ such that
\begin{equation*}\label{eqpent}
s_{23}\, s_{13}\, s_{12}=s_{12}\, s_{23},
\end{equation*}
where $s_{12}=s\times \id_M$,  $s_{23}=\id_M\times s$ and  $s_{13}=(\id_M\times \tau)\,s_{12}\,(\id_M\times \tau)$, and $\tau$ is the flip map, i.e., the permutation on $M\times M$ given by $\tau(x,y)=(y,x)$, for all $x,y\in M$.\\ 
Instead, a \emph{set-theoretical solution of the reversed pentagon equation} on the set $M$ is a map $s:M\times M\longrightarrow M\times M$ such that
\begin{equation*}
s_{12}\, s_{13}\, s_{23}=s_{23}\, s_{12}.
\end{equation*}
Hereinafter, by \emph{solution} we mean a set-theoretical solution of the pentagon equation, whereas by \emph{reversed solution} we mean a set-theoretical solution of the reversed pentagon equation.
\end{defin}

We remark that a map $s:M\times M\longrightarrow M\times M$ is a solution if and only if $\tau s\tau$ is a reversed solution. Moreover, $s$ is an invertible solution if and only if $s^{-1}$ is a reversed solution.\\
Note that the identity map $\id_{M\times M}$ is a solution (and a reversed solution) on $M$, but the flip map $\tau$ is not a solution when $|M|>1$.

Here, we give some easy examples of solutions. We point out that some of the solutions are obtained in non-algebraic contexts.
\begin{exs}\label{exsKey}\hspace{1mm}\\
    	(1)\; (\emph{Kac-Takesaki's solutions}) Let $(M,\cdot)$ be a group. The maps $s,t:M \times M \rightarrow M \times M$ given by 
	    $$ s(x,y)=(x\cdot y, y) , \quad t(x,y)=(x, x^{-1}\cdot y),$$
	    for all $x,y \in M$, are invertible solutions on $M$.
	    
	    \vspace{2mm}\noindent
	    (2) \; Let $(M,\cdot)$ be a semigroup and $\gamma$ an idempotent endomorphism of $M$. Then, the map $s:M \times M \rightarrow M \times M$ given by 
        $$ s(x,y)=\left(x\cdot y,\gamma\left(y\right)\right),$$ 
	    for all $x,y \in M$, is a solution on $M$. 
	    In particular, if $e$ is an idempotent element of $(M,\cdot)$, then $s\left(x,y\right)= \left(x\cdot y,\; e\right)$ is a solution on $M$.
	    
	    \vspace{2mm}\noindent 
		(3)\; (\emph{Militaru's solution}, \cite{Mi98}) If $M$ is a set and $\alpha,\beta$ are idempotent maps from $M$ to itself such that $\alpha\beta=\beta\alpha$, then the map $s:M \times M \rightarrow M \times M$ given by 
		$$s(x,y)=\left(\alpha\left(x\right),\, \beta\left(y\right)\right),$$
		for all $x, y\in M$, is a solution and a reversed solution on $M$. 
	
        \vspace{2mm}
		Let $M$ be a group and $A,B$ two subgroups such that $A\cap B=\{1\}$ and $M=AB$. 
		Then, for every $x\in M$ there exists a unique couple $(a,b)\in A\times B$ such that $x=ab$.
		Let  $p_1:M \rightarrow A$ and  $p_2: M \rightarrow B$ be maps such that $x=p_1(x)\,p_2(x)$, for every $x\in M$.
		
		\vspace{2mm}\noindent 
    	(4) (\emph{Zakrzewski's solution}, \cite{Za92})  The map $s:M\times M\longrightarrow M\times M$ given by 
		$$
			s\left(x,y\right)=\left(p_2\left(yp_1\left(x\right)^{-1}\right)x,\; yp_1\left(x\right)^{-1}\right),
		$$
		for all $x, y \in M$, is an invertible solution on $M$.
		
		\vspace{2mm}\noindent 
		(5) (\emph{Baaj-Skandalis' solution}, \cite{BaSk98})  The map $s:M\times M\longrightarrow M\times M$ given by 
		$$
		s\left(x,y\right)=\left(xp_1\left(p_2\left(x\right)^{-1}y\right), \, p_2\left(x\right)^{-1}y\right)
		$$
		for all $x, y \in M$, is an invertible solution on $M$.
\end{exs}

In order to obtain different solutions, we give the classic definition of equivalent solutions below.
\begin{defin} Let $M, N$ be two sets, $s$ and $r$ solutions on $M$ and $N$ respectively. Then, $r$ and $s$ are said to be \textit{equivalent} if there exists a a bijection $\eta:M\rightarrow N $ such that $(\eta \times \eta)\;s= r \;(\eta \times \eta)$.
\end{defin}

As noted above, we remark that $s$ is a solution on a set $M$ if and only if $\tau s^{-1}\tau$ is also a on $M$. According to the definition of opposite operator of a multiplicative unitary operator given by Baaj and Skandalis \cite{BaSk93}, we give the following definition for a solution on an arbitrary set.  

\begin{defin} Let $s$ be an invertible solution on a set $M$. The solution $s^{op}:=\tau s^{-1}\tau$ on $M$ is called the \emph{opposite solution} of $s$.
\end{defin}

\begin{exs}\hspace{1mm}
	\begin{enumerate}
		\item  If $s$ and $t$ are Kac-Takesaki's solutions, then the opposite solutions of $s$ and $t$ are respectively
		\begin{center}
			$s^{op}(x,y)=(x,\; yx^{-1})$ \quad and \quad $t^{op}(x,y)=(yx,\; y)$.
		\end{center}
		\item  Zakrzewski and Baaj-Skandalis' solutions are opposite of each other.
	\end{enumerate}
\end{exs}

\vspace{3mm}

Baaj and Skandalis \cite{BaSk93} introduced  the concepts of commutativity and cocommutativity for multiplicative unitary operators defined on a Hilbert space. We translate these definitions for solutions on an arbitrary set. 
\begin{defin}
	Let $M$ be a set and $s$ a solution on $M$. Then, $s$ is said to be
	\begin{enumerate}
		\item \emph{commutative} if $s_{12}s_{13}=s_{13}s_{12}$, 
		\item \emph{cocommutative} if $s_{13}s_{23}=s_{23}s_{13}$.
	\end{enumerate} 
\end{defin}

According to Militaru \cite{Mi98}, a \emph{reversed solution} is commutative [resp. cocommutative] if the solution $\tau s\tau$ is commutative [resp. cocommutative]. Furthermore, if $s$ is an invertible solution on a set $M$, then $s$ is commutative [resp. cocommutative] if and only if $s^{op}$ is cocommutative [resp. commutative]. 

\vspace{3mm}

\begin{exs}\hspace{1mm}
	\begin{enumerate}
		\item Let $s$ and $t$ be Kac-Takesaki's solutions. Then, $s$ is cocommutative, whereas $t$ is commutative.
		\item Militaru's solutions are both commutative and cocommutative.
	\end{enumerate}
\end{exs}

\vspace{3mm}

\section{Description of the solutions on a group}
 In order to give a complete description of the solutions $s(x,y)=(x\cdot y\; , \; x\ast y)$ on a group $(M,\cdot)$, we slightly change this notation. We write $xy$ instead of $x \cdot y$ and $\theta_x(y)$ instead $x\ast y$, for all $x,y \in M$. Therefore, we give the following characterization for a solution $s$ on a set $M$, of which the proof is a routine computation. 
\begin{prop}\label{STR} Let $M$ be a set, $s: M \times M \rightarrow M \times M$ 
	a map and write
	\begin{center}
		$s(x,y)=\left(xy, \theta_x\left(y\right)\right)$,
	\end{center}
	where $\theta_x: M \rightarrow M$, for every $x \in M$. Then, $s$ is a solution on $M$ if and only if the following conditions hold
	\begin{align}
	(xy)z&=x(yz),\label{assoc}\\
	\theta_x(y)\theta_{xy}(z)&=\theta_x(yz),\label{omotheta}\\
	\theta_{\theta_x(y)}\theta_{xy}&=\theta_y,\label{idemtheta}
	\end{align}
	for all $x,y,z \in M$.
	Moreover, s is invertible if and only if for any pair $(x,y)\in M \times M$ there exists a unique pair $(u,z) \in M \times M$ such that
	\begin{center}
		$uz =x$, \quad $\theta_u(z)=y$.
	\end{center}
\end{prop}

\vspace{3mm}

Now, we present a useful tool to construct new solutions on a group, by examining its normal subgroups.
\begin{prop}\label{propmu}
	Let $M$ be a group, $K$ a normal subgroup of $M$, and $R$ a system of representatives of $M/K$ such that $1 \in R$.  If $\mu: M \rightarrow R$ is a map such that $\mu(x) \in Kx$, for every $x \in M$, then the map $s:M\times M\longrightarrow M\times M$ given by 
	\begin{equation*}
		s(x,y)=\left(xy, \mu\left(x\right)^{-1}\mu\left(xy\right)\right),
	\end{equation*}
 for all $x,y \in M$, is a solution on $M$.
\end{prop}
\begin{proof}
In order to prove that $s$ is a solution, we shall prove only \eqref{idemtheta} of \cref{STR}, since conditions (1) and \eqref{omotheta} are straightforward. By hypothesis, if $x\in M$, since $K$ is a normal subgroup of $M$, it follows that $\mu(x)^{-1} \in Kx^{-1}$. Now, we compare
\begin{align*}
\theta_{\theta_x(y)}\theta_{xy}(z):&= \mu\left(\mu\left(x\right)^{-1} \mu\left(x y\right)\right)^{-1}\mu\left(\mu\left(x\right)^{-1} \mu\left(xy\right)  \mu\left(xy\right)^{-1}  \mu\left(xyz\right)\right)\\
&=\mu\left(\mu\left(x\right)^{-1} \mu\left(xy\right)\right)^{-1} \mu\left(\mu\left(x\right)^{-1}  \mu\left(x y z\right)\right)
\end{align*}
and 
\begin{center}
$\theta_y(z):=\mu(y)^{-1} \mu(y z)$. 
\end{center}
Note that
\begin{center}
	$\mu\left(\mu\left(x\right)^{-1} \mu\left(x y\right)\right)^{-1} \in K\left(\mu\left(x\right)^{-1}  \mu\left(x  y\right)\right)^{-1}=K\left(\mu\left(x y\right)^{-1} \mu(x)\right)$
\end{center}
and $\mu(xy)^{-1} \mu(x) \in Ky^{-1}$. Thus,
\begin{center}
	$\mu\left(\mu\left(x\right)^{-1}\mu\left(x y\right)\right)^{-1} \in Ky^{-1}$.
\end{center}
Hence, since also $\mu(y)^{-1} \in Ky^{-1}$, we get \begin{center}
$\mu\left(\mu\left(x\right)^{-1} \mu\left(x y\right)\right)^{-1}=\mu(y)^{-1}$.
\end{center}
Analogously, one can prove that $\mu\left(\mu\left(x\right)^{-1}\mu\left(x y z\right)\right)$ and $\mu(y)^{-1} \mu(y z)$ are elements of $K(y z)$ and so they shall be equal. Hence, $\theta_x(y)\theta_{xy}(z)=\theta_y(z)$. \\
Therefore, $s(x,y)=\left(xy, \mu\left(x\right)^{-1}\mu\left(xy\right)\right)$ is a solution on $M$.
\end{proof}

\vspace{3mm}

Thus, a lot of examples of solutions on groups may be obtained using the \cref{propmu} above. Some very easy examples of solutions on the symmetric groups are the following.

\begin{ex} Let $n \geq 3$, $\mathcal{S}_n$ the symmetric group of order $n$, $\mathcal{A}_n$ the alternating group of degree $n$, and $R=\lbrace 1, \pi \rbrace$ a system of representatives of $\mathcal{S}_n/\mathcal{A}_n$, where $\pi$ is a transposition of $\mathcal{S}_n$. Then, the map $\mu: \mathcal{S}_n \to R$ given by
\begin{center}
 $\mu(\alpha)=\begin{cases}
   \pi \qquad &\text{if} \;\sgn(\alpha)=-1 \\
   \id_{\mathcal{S}_n} \quad&\text{if}\; \sgn(\alpha)=1
   \end{cases},$
\end{center}
for every $\alpha \in \mathcal{S}_n$, satisfies the hypothesis of Proposition \ref{propmu}. Therefore, the map given by $s(\alpha,\beta)=\left(\alpha\beta,\mu\left(\alpha\right)^{-1}\mu\left(\alpha\beta\right)\right)$ is a solution on $\mathcal{S}_n$.
\end{ex}

\vspace{3mm}

Now, we want to show that all the solutions on a group $M$ are obtained as in \cref{propmu}.
For this purpose, we need to start by proving several lemmas, in which, for every $x\in M$, interesting properties of the map $\theta_x$ are studied and, in particular, of the map $\theta_1$. Indeed, if $s(x,y)=(xy, \,\theta_x(y))$ is a solution on a group $M$, by condition \eqref{omotheta} in \cref{STR}, we get
\begin{equation}\label{thetax}
	\theta_x(y)=\theta_1(x)^{-1}\theta_1(xy),
\end{equation}
for all $x, y \in M$. 

\begin{lemma}\label{proptheta} Let $s(x,y)=(xy, \,\theta_x(y))$ be a solution on a group $M$. Then, the following conditions hold
\begin{enumerate}
\item$\theta_x(1)=1$,
\item$\theta_1\theta_x=\theta_1$,
\item $\theta_1(x)^{-1}=\theta_x(x^{-1})$,
\end{enumerate}
for every $x \in M$.
\begin{proof} The statements easily follow by conditions \eqref{assoc}-\eqref{idemtheta} in \cref{STR}.
Indeed, by \eqref{omotheta}, we obtain $\theta_x(1) \; \theta_x(1) = \theta_x(1)$ and so $\theta_x(1)=1$, for every $x \in M$. Moreover, by condition \eqref{idemtheta}, we obtain
 $\theta_1\theta_x=\theta_{\theta_x(1)}\theta_{x}=\theta_1$.
Finally, by \eqref{omotheta}, we get
\begin{center}
$\theta_x(x^{-1})\;\theta_1(x)=\theta_x(x^{-1})\;\theta_{x x^{-1}}(x)=\theta_x(x^{-1}  x)=\theta_x(1)=1$.
\end{center}
Therefore, $\theta_1(x)^{-1}=\theta_x(x^{-1})$.
\end{proof}
\end{lemma}

\begin{remk} 
	In general, we observe that $\theta_1$ is not a homomorphism. \\
	For example, let $M=\left\langle a\right\rangle$ be the cyclic group of order $6$ and $\theta_1:M \rightarrow M$ the map given by
\begin{center}
	$\theta_1(1)=1$, \; $\theta_1(a)=a$, \; $\theta_1(a^2)=1$, \;$\theta_1(a^3)=a$, \; $\theta_1(a^4)=1$, \; $\theta_1(a^5)=a$.
\end{center}
Considering \eqref{thetax}, we have that $\theta_x(y)=\theta_1(x)^{-1}\theta_1(xy)$, for all $x,y \in M$, and then one can see that $s(x,y)=\left(xy,\theta_x(y)\right)$ is a solution on $M$. But, for instance,
$\theta_1(a^4)=1$, whereas $\theta_1(a)\;\theta_1(a^3)=a^2$. Therefore, $\theta_1$ is not a homomorphism.
\end{remk}

\noindent In spite of that, we are able to show the following result.

\begin{lemma} \label{Knormal}
Let $s(x,y)=\left(xy, \theta_x(y)\right)$ be a solution on a group $M$. Then, the subset of $M$
\begin{center}
$K:= \lbrace x \; |\; x \in M, \;\theta_1(x)=1 \rbrace$,
\end{center}
is a normal subgroup of $M$, called the \emph{kernel of $s$}.
\begin{proof}
If $x, y \in K$, by \eqref{omotheta} in \cref{STR}, we have that
$$\theta_1(x y)=\theta_1(x) \;\theta_x(y)=\theta_x(y).$$
Furthermore, $\theta_x(y) \in K$. Indeed, by 2. in \cref{proptheta}, we get 
\begin{center}
$\theta_1(\theta_x(y))=\theta_1(y)=1$
\end{center} 
and, since $\theta_1$ is idempotent, we obtain
\begin{center}
$\theta_1(xy)=\theta_1\theta_1(x y)=\theta_1(\theta_x(y))=1$.
\end{center}
Hence, $x y$ is in $K$. Moreover, if $x \in K$, then $\theta_x(x^{-1})=\theta_1(x)^{-1}=1$ and so, by 2. in \cref{proptheta}, we have
\begin{center}
$\theta_1(x^{-1})=\theta_1\theta_x(x^{-1})=\theta_1(1)=1$.
\end{center}
Therefore, $x^{-1}$ is in $K$ and so $K$ is a subgroup of $M$. \\
Finally, we prove that $K$ is normal, that means $\theta_1(x^{-1}  k  x)=1$, for every $x \in M$ and $k \in K$. To this end, we observe that
\begin{equation}\label{thetak}
\theta_x(k)=1.
\end{equation} 
Indeed, by \eqref{idemtheta} in \cref{STR} and by 1. in \cref{proptheta}, we get
\begin{center}
$\theta_x(k)=\theta_{\theta_{x^{-1}}(x)}\theta_{x^{-1} x}(k)=\theta_{\theta_{x^{-1}}(x)}\theta_1(k)=\theta_{\theta_{x^{-1}}(x)}(1)=1$.
\end{center}
Moreover,
\begin{equation}\label{thetakx}
\theta_k(x)=\theta_1(x),
\end{equation}
because $\theta_k(x)=\theta_{\theta_1(k)} \theta_k(x)=\theta_1\theta_k(x)=\theta_1(x)$.\\
Now, by \eqref{idemtheta} in \cref{STR} and condition \eqref{thetakx} above, we obtain
\begin{align*}
\theta_{x^{-1}k}\left(x\right)&=\theta_{\theta_x\left(x^{-1} k\right)}\theta_{x x^{-1}  k}\left(x\right)=\theta_{\theta_x\left(x^{-1}\right)  \theta_1\left(k\right)}\theta_k\left(x\right)=\theta_{\theta_x\left(x^{-1}\right)}\theta_k\left(x\right)\\
&=\theta_{\theta_x(x^{-1})}\theta_1(x)=\theta_{x^{-1}}(x)
\end{align*}
and, since $\theta_{x^{-1}}(x)=\theta_1(x^{-1})^{-1}$ by 3. in \cref{proptheta}, we obtain
\begin{align*}
\theta_1(x^{-1}  k \ x)&=\theta_1(x^{-1})  \theta_{x^{-1}}(k  x)= \theta_1(x^{-1}) \theta_{x^{-1}}(k) \theta_{x^{-1}  k}(x)=\theta_1(x^{-1}) \theta_{x^{-1}}(x)\\
&=\theta_1(x^{-1}) \; \theta_1(x^{-1})^{-1}=1.
\end{align*}
Therefore, $K$ is a normal subgroup of $M$.
\end{proof}
\end{lemma}

\vspace{3mm}

In the following result, we prove some interesting properties of the kernel $K$ of a solution $s$ on a group $M$.

\begin{lemma}\label{propkern}
Let  $s(x,y)=(x y, \; \theta_x(y))$ be a solution on a group $M$ and $K$ the kernel of $s$. Then, the following conditions
\begin{enumerate}
\item $\theta_x(ky)=\theta_x(y)$,
\item $\theta_{kx}(y)=\theta_x(y)$,
\item$\theta_1(Kx) \subseteq Kx $,
\end{enumerate}
hold, for all $x,y \in M$ and for every $k \in K$.
\begin{proof} Let $x, y \in M$ and $k \in K$.
\begin{enumerate}
\item Considering \eqref{thetak} we have that $\theta_x(k)=1$, and by condition \eqref{thetakx} in \cref{Knormal}, we get
\begin{center}
$\theta_x(ky)=\theta_x(k) \theta_{xk}(y)=\theta_{\theta_{x^{-1}}(xk)}\theta_k(y)=\theta_{\theta_{x^{-1}}(x)\theta_1(k)}\theta_1(y)=\theta_x(y).$
\end{center}
\item By \eqref{omotheta} in \cref{STR} and condition \eqref{thetakx} in \cref{Knormal}, we have 
\begin{center}
$\theta_{kx}(y)=\theta_k(x)^{-1} \theta_k(x y)=\theta_1(x)^{-1}\theta_1(x y)=\theta_x(y)$.
\end{center}
\item Finally, we show that $x^{-1}\theta_1(kx) \in K$. But, since $\theta_1(kx)=\theta_1(x)$, we have to prove that
\begin{center}
$\theta_1\left(x^{-1}\theta_1\left(x\right)\right)=1$.
\end{center}
By condition 2. in \cref{proptheta}, we compute
\begin{align*}
\theta_1\left(x^{-1}\theta_1\left(x\right)\right)&=\theta_1 \theta_x\left(x^{-1}\theta_1\left(x\right)\right)=\theta_1 \left(\theta_x\left(x^{-1}\right) \theta_{xx^{-1}}\theta_1\left(x\right)\right)=\theta_1 \left(\theta_x\left(x^{-1}\right) \theta_1\left(x\right)\right)\\
&=\theta_1\left(\theta_1\left(x\right)^{-1}\theta_1\left(x\right)\right)=1.
\end{align*}
Therefore, $\theta_1(xk)$ is in $Kx$.
\end{enumerate}
\end{proof}
\end{lemma}

Now, we are able to prove the main result of this paper, in which we prove the converse of \cref{propmu}.

\begin{theor}\label{teof}
		Let $M$ be a group, $K$ a normal subgroup of $M$, and $R$ a system of representatives of $M/K$ such that $1 \in R$.  If $\mu: M \rightarrow R$ is a map such that $\mu(x) \in Kx$, for every $x \in M$, then the map $s:M\times M\longrightarrow M\times M$ given by 
	\begin{equation*}
	s(x,y)=\left(xy, \mu\left(x\right)^{-1}\mu\left(xy\right)\right),
	\end{equation*}
	for all $x,y \in M$, is a solution on $M$.\\
    Conversely, if $s(x,y)=(xy,\theta_x(y))$ is a solution on $M$, then there exists a normal subgroup $K$ of $M$ such that $\theta_1(M)$ is a system of representatives of $M/K$, $1 \in \theta_1(M)$, $\theta_1(x) \in Kx$, for every $x\in M$, and 
    \begin{equation*}
    s(x,y)=\left(xy, \ \theta_1\left(x\right)^{-1}\theta_1\left(xy\right)\right),
    \end{equation*}
    for all $x, y \in M$.
\end{theor}

\begin{proof}The first part is proved in \cref{propmu}. Now, let $s(x,y)=\left(xy,\theta_x\left(y\right)\right)$ be a solution on $M$. In order to prove the converse, we recall that by condition \eqref{thetax}
\begin{equation*}
\theta_x(y)= \theta_1(x)^{-1} \; \theta_1(x  y),
\end{equation*}
for all $x, y \in M$. Moreover, let $K$ be the kernel of $s$, that is a normal subgroup of $M$, as we proved in \cref{Knormal}. Then, $\theta_1(M)$ is a system of representatives of $M/K$ that contains $1$. Indeed, $1=\theta_1(1)$ by 1. \cref{proptheta}. Furthermore, by 3. in \cref{propkern}, we have that
\begin{center}
$\theta_1(x)\in \theta_1(M) \cap Kx$,
\end{center} 
for every $x \in M$. Now, if $\theta_1(y)$ is another element of $\theta_1(M) \cap Kx$, since $\theta_1(y) \in  Ky$, we get $xy^{-1}\in K$. It follows that there exists $k \in K$ such that $x=ky$. Hence, by 1. in \cref{propkern}, $\theta_1(x)=\theta_1(ky)=\theta_1(y)$. So, $\theta_1(M)$ is a system of representatives of $M/K$ and  $\theta_1(x) \in Kx$, for every $x\in M$.
\end{proof}

\vspace{3mm}

As an immediate consequence of \cref{teof}, if $M$ is a simple group, the only solutions on $M$ are $s(x,y)=(xy,y)$ and $s(x,y)=(xy,1)$. Furthermore, as Kashaev and Sergeev observed in \cite{KaSe98}, if $s$ is an invertible map, the only solution on $M$ is given by $\theta_x(y)=y$. We give another proof of this fact, by using \cref{teof}.
\begin{cor}
	Let $M$ be a group and $s(x,y)=(xy, \theta_x(y))$ an invertible solution on $M$. Then, $\theta_x=\id_M$ holds, for every $x \in M$.
\end{cor}
\begin{proof}
	By \cref*{teof}, if $K$ is the kernel of $s$, then $\theta_x(y) \in Ky$, for all $x,y \in M$. Hence, there exists $k \in K$ such that $\theta_x(y)=ky$, for all $x,y \in M$ and, by 2. in \cref{proptheta}, we get
	\begin{center}
		$\theta_1(y)=\theta_1\theta_x(y)=\theta_1(ky)$.
	\end{center}
	Then, since $\theta_1$ is injective, $ky=y$, for every $y \in M$. Therefore, $\theta_x=\id_M$, for every $x\in M$.
\end{proof}

\vspace{3mm}

\section{Some comments and questions}
In this section we raise some natural questions about set-theoretical solutions of pentagon equation. \\
The first natural question is to describe all the solutions $s(x,y)=(x \cdot y, x\ast y)$ when $(M, \ast)$ is a group. Here, for convenience, we set $xy:=x \ast y$, for all $x,y, \in M$.

\begin{prop}Let $M$ be a group. Then, $s(x,y)=(x \cdot y, xy)$ is a solution on $M$ if and only if $M$ is an elementary abelian $2$-group and $x \cdot y=x$ holds, for all $x, y \in M$.
\end{prop}
\begin{proof}
Assuming that $s(x,y)=(x \cdot y, xy)$ is a solution on $M$, we obtain
	$xy(x \cdot y)z=yz$, for all $x,y,z \in M$, and so this implies
	\begin{equation}\label{cdot}
	x \cdot y=y^{-1}x^{-1}y.
	\end{equation}
	Furthermore, $xy \cdot \left(\left(x \cdot y\right )z\right)=x(y \cdot z)$ and, by \eqref{cdot}, this equality becomes
	\begin{center}
		$z^{-1}y^{-1}y^{-1}x^{-1}=xz^{-1}y^{-1}$.
	\end{center}
	In particular, if $y=z=1$, we get $x^{-1}=x$. In addition, from the associativity law of $\cdot$, we have 
	\begin{center}
		$y^{-1}xy=yzx^{-1}z^{-1}y^{-1}$.
	\end{center}
	Hence, $xz=zx$, for all $x,z \in M$. Therefore, $M$ is an elementary abelian $2$-group and $x \cdot y=x$, for all $x,y, \in M$.\\
	Conversely, if $x,y, z \in M$ we have
	\begin{align*}
	s_{23}s_{13}s_{12}(x,y,z)&=(x,xy,xyxz)=(x,xy,yz)=s_{12}s_{23}(x,y,z)
	\end{align*}
	and thus $s$ is a solution on $M$.
\end{proof}

\vspace{3mm}

As we noted above, Militaru's solutions are also reversed solutions. We can also build other solutions of this type, but we are not able to  describe them all. Here, we characterize only the solutions on groups that are also reversed solutions. In particular, if $(M, \ast)$ is an elementary abelian $2$-group, it is easy to check that the solution $s(x,y)=(x,x\ast y)$ on $M$ is also reversed. Instead, by \cref{teof}, in the case $(M,\cdot)$ is a group, we can prove the following result.

\begin{prop}Let $M$ be a group and $s(x,y)=\left(xy, \theta_x(y)\right)$ a solution on $M$. Then, $s$ is a reversed solution on $M$ if and only if $M$ is an elementary abelian $2$-group and $\theta_x=\id_M$ holds, for every $x \in M$.
\end{prop}
\begin{proof}
	First, suppose $s(x,y)=\left(xy, \theta_x\left(y\right)\right)$ is both a solution and a reversed solution on $M$. Moreover, let $K$ be the kernel of $s$. Then, by \cref{teof}, $\theta_x(y)\in Ky$, for all $x,y \in M$. Hence, requiring that $s$ is also a reversed solution, we get $kyxy=x$, for all $x,y \in M$ and for some $k \in K$.
	In particular, if $x=1$ and $y=1$, then $k=1$. So, $M$ is an elementary abelian $2$-group and  $\theta_x(y)=y$ holds, for all $x,y \in M$.\\
	Conversely, we suppose $M$ is an elementary abelian $2$-group. Then, $s(x,y)=\left( xy,y\right)$ is a solution on $M$ and
	\begin{align*}
	s_{12}s_{13}s_{23}(x,y,z)=(xzyz,yz,z)=(xy,yz,z)=s_{23}s_{12}(x,y,z).
	\end{align*}
	Therefore, $s$ is also a reversed solution on $M$. 
\end{proof}

\vspace{3mm}

Finally, we are still not able to characterize solutions on a set $M$ that are both commutative and cocommutative. We completely describe this type of solutions in the specific case when either $(M, \ast)$ or $(M, \cdot)$ is a group. Indeed, in the first case, if $(M, \ast)$ is an elementary abelian $2$-group, the solution $s(x,y)=(x,x\ast y)$ is both commutative and cocommutative; whereas if $(M,\cdot)$ is a group, we give the following result.

\begin{prop}
Let $M$ be a group and $s(x,y)=\left(xy, \theta_x(y) \right)$ a solution on $M$. Then, $s$ is both commutative and cocommutative if and only if $M$ is abelian and $\theta_x=\id_M$ holds, for every $x \in M$.
\end{prop}
\begin{proof}
	Suppose the solution $s(x,y)=\left(xy, \theta_x(y) \right)$ is both commutative and cocommutative. Then, since $s$ is commutative, $M$ is abelian; moreover, since $s$ is cocommutative, $\theta_x(y)=y$ holds, for all $x,y \in M$.\\
	Conversely, if $M$ is abelian, the solution $s(x,y)=(xy,y)$ is trivially both commutative and cocommutative.
\end{proof}

\vspace{5mm}

\end{document}